\documentclass[onefignum,onetabnum]{siamart220329}

\usepackage{graphicx}
\usepackage{amsmath}
\usepackage{amsfonts}
\usepackage{enumitem}
\usepackage[title]{appendix}

\usepackage{cleveref}

\crefname{equation}{Eq.}{Eqs.}
\Crefname{equation}{Equation}{Equations}
\crefname{figure}{Fig.}{Figs.}
\Crefname{figure}{Figure}{Figures}

\newsiamremark{remark}{Remark}
\newsiamremark{hypothesis}{Hypothesis}
\crefname{hypothesis}{Hypothesis}{Hypotheses}
\newsiamthm{claim}{Claim}

\DeclareMathOperator*{\argmin}{arg\,min}

\newcommand{\qbox}[1]{\quad\hbox{#1}\quad}

\newcommand{\R}{\mathbb{R}}

\title{ Adaptive Accelerated Gradient Method for Smooth Convex Optimization
\thanks{Submitted to the editors DATE.
\funding{This work was partially funded by the China Scholarship Council~202208520010, and also benefited from the support of the FMJH Program Gaspard Monge for optimization and operations research and their interactions with data science.}}}

\author{Zepeng Wang\thanks{Bernoulli Institute for Mathematics, Computer Science and Artificial Intelligence, University of Groningen, Groningen, The Netherlands 
  (\email{zepeng.wang@rug.nl}).}
\and Juan Peypouquet\thanks{Bernoulli Institute for Mathematics, Computer Science and Artificial Intelligence, University of Groningen, Groningen, The Netherlands 
  (\email{j.g.peypouquet@rug.nl}).}
}

\begin{document}
\maketitle

% REQUIRED
\begin{abstract}
We propose an adaptive accelerated gradient method for solving smooth convex optimization problems. The method incorporates a scheme to determine the step size adaptively, by means of a local estimation of the smoothness constant, which is assumed unknown, without resorting to line search procedures. The sequence generated by this method converges weakly to a minimizer of the objective function, and the function values converge at a fast rate of $\mathcal{O}\left( \frac{1}{k^2} \right)$. Moreover, if the objective function is strongly convex, the function values converge at a linear rate.
\end{abstract}

% REQUIRED

% REQUIRED

%%%%%%%%%%%%%%%%%%%%%%%%%%%%%%%%%%%%%%%%
%%%%%%%%%%%%%%%%%%%%%%%%%%%%%%%%%%%%%%%%
\section{Introduction}
Let $H$ be a real Hilbert space and $f:H\to\R$ be convex and $L$-smooth, for some $L>0$. In this paper, we are interested in the following minimization problem:
\begin{equation}\label{Problem: min_f}
\min_{x\in H} f(x),
\end{equation}
and write $x^*=\argmin(f)$, $f^*=f(x^*)$ hereafter for simplicity. A basic method to solve this problem is gradient descent:
\begin{equation}\label{Algo: GD}
x_{k+1} = x_k - s\nabla f(x_k),
\end{equation} 
where $s>0$ is the step size. If the global smoothness parameter $L$ is known, one can set $s\in\left(0,\frac{2}{L}\right)$ and obtain $f(x_k)-f^*\le \mathcal{O}\left( \frac{1}{k} \right)$ \cite{Nesterov_2004}. But in practice, $L$ is usually unknown and requires some estimation procedures, such as line search \cite{Armijo_1966}, which can be computationally expensive for large scale problems. 

As an alternative, the authors in \cite{Malitsky_2020} leveraged a local estimate of $L$, namely
$$ L_k = \frac{ \| \nabla f(x_k) - \nabla f(x_{k-1}) \| }{ \| x_k - x_{k-1} \| }, $$
and computed the current step size in relation to this estimate, by
$$ s_k = \min\left\{ \sqrt{1+\frac{s_{k-1}}{s_{k-2}}} s_{k-1}, \frac{1}{2L_k} \right\}. $$
The result is an {\it adaptive} variant of \eqref{Algo: GD} that can guarantee a convergence rate of $\mathcal{O}\left( \frac{1}{k} \right)$ for the function values. Since the smoothness estimate satisfies $L_k\le L$, a larger step size than $\frac{1}{L}$ is allowed in principle and can speed up the convergence rate. The interested reader is referred to \cite{Ma_2025_AdaBB,Ghaderi_2025} for further progress (including the convergence of the iterates to an optimal solution) on the adaptive gradient method for smooth functions, and \cite{Malitsky_2024,Latafat_2025,Ma_2025_AdaBB} for the composite case. If the objective function is smooth and {\it strongly convex}, a linear convergence rate was established in \cite{Ma_2025_AdaBB,Ghaderi_2025}.

An inertial variant of \eqref{Algo: GD} is Nesterov's accelerated gradient method \cite{Nesterov_1983}:
\begin{equation}\label{Algo: NAG}
\left\{
\begin{array}{rcl}
y_{k+1} &=& x_k - s\nabla f(x_k),\\[3pt]
x_{k+1} &=& y_{k+1} + \frac{\theta_k-1}{\theta_{k+1}}(y_{k+1}-y_k),
\end{array}
\right.
\end{equation}
where $s>0$ and $(\theta_k)_{k\ge 0}$ is given by
\begin{equation}\label{E: theta_k}
\theta_k = \left\{
\begin{array}{ccl}
1, &\qbox{if}& k=0,\\[3pt]
\frac{1 + \sqrt{1+4\theta_{k-1}}}{2}, &\qbox{if}& k\ge 1.
\end{array}
\right.
\end{equation} 
If the global smoothness parameter $L$ is known, one can set $s\in\left(0,\frac{1}{L}\right]$ and obtain $f(y_{k+1}) - f^* \le \mathcal{O}\left( \frac{1}{k^2} \right)$ \cite{Nesterov_1983}, which improves upon the rate of gradient descent \eqref{Algo: GD}. A linear convergence rate was obtained for strongly convex functions \cite{Shi_2024,Bao_2023,Wang_2025_AVD}. Convergence of the iterates was proved in \cite{Chambolle_2015,Attouch_2018,Ryu_2025,Radu_2025_iterate}.   

Inspired by \cite{Malitsky_2020}, there has been a growing interest to develop an adaptive variant of \eqref{Algo: NAG} which is line-search-free but still preserves a fast convergence rate $\mathcal{O}\left( \frac{1}{k^2} \right)$. The first adaptive accelerated gradient algorithm to achieve this goal was developed by \cite{Lan_2025}, followed by \cite{Ma_2025} and \cite{Borodich_2025}. The method in \cite{Ma_2025}, upon which this work is based, can be rewritten as
\begin{equation}\label{Algo: A-OGM}
\left\{
\begin{array}{rcl}
y_{k+1} &=& x_k - s_k\nabla f(x_k),\\[3pt]
x_{k+1} &=& y_{k+1} + \frac{\theta_{k+2}-1}{\theta_{k+3}}( y_{k+1} - y_k ) + \frac{\theta_{k+2}}{\theta_{k+3}}(\alpha_k-1)(y_{k+1}-x_k),
\end{array}
\right.
\end{equation}
where $(\theta_k)_{k\ge 0}$ is given by \eqref{E: theta_k}, $(s_k)_{k\ge 0}$ is the (adaptive) step size and $(\alpha_k)_{k\ge 0}$ is a suitable positive sequence. If $\alpha_k\equiv 1$, \eqref{Algo: A-OGM} reduces to \eqref{Algo: NAG}. Under these assumptions, one has $f(x_k)-f^*\le\mathcal{O}\left( \frac{1}{k^2} \right)$ and $\min_{i\in\{0,\cdots,k\}} \| \nabla f(x_i) \|^2 \le \mathcal{O}\left( \frac{1}{k^3} \right)$ \cite{Ma_2025}. In this paper, we revisit the behavior of the iterates generated by \eqref{Algo: A-OGM}, with different parameter choices. On the one hand, we show the previously unknown weak convergence of the iterates to a solution of \eqref{Problem: min_f}. On the other (and more importantly!), we prove that linear convergence holds when the algorithm is applied to strongly convex functions.  

The remainder of the paper is organized as follows. In Section \ref{Sec: algorithm}, we present an overview of the adaptive accelerated gradient method. In Section \ref{Sec: estimations}, we introduce an energy sequence, and derive some preliminary estimations. In Section \ref{Sec: convexity}, we provide an alternative proof of the $\mathcal{O}\left( \frac{1}{k^2} \right)$ convergence rate, and use the arguments to prove the convergence of the iterates. In Section \ref{Sec: strong convexity}, we derive a linear convergence rate for the function values under strong convexity. Some conclusions are given in Section \ref{Sec: conclusions}.

%%%%%%%%%%%%%%%%%%%%%%%%%%%%%%%%%%%%%%%%
%%%%%%%%%%%%%%%%%%%%%%%%%%%%%%%%%%%%%%%%
\section{Adaptive accelerated gradient method}\label{Sec: algorithm}
In this section, we present an overview of the method and comment on the choices of the parameter sequences.

Consider the Adaptive Accelerated Gradient Method:
\begin{equation}\label{Algo: A-AGM}\tag{AdaAGM}
\left\{
\begin{array}{rcl}
y_{k+1} &=& x_k - s_k\nabla f(x_k),\\[4pt]
x_{k+1} &=& y_{k+1} + \frac{t_k - 1}{t_{k+1}}( y_{k+1} - y_k ) + \frac{(\gamma-1)t_k}{t_{k+1}}( y_{k+1}-x_k ),\\[4pt]
L_{k+1} &=& \frac{ \frac{1}{2}\| \nabla f(x_{k+1}) - \nabla f(x_k) \|^2 }{ \langle \nabla f(x_{k+1}), x_{k+1} - x_k \rangle - \left( f(x_{k+1}) - f(x_k) \right) },\\[4pt]
s_{k+1} &=& T(s_k, L_{k+1}),
\end{array} 
\right.
\end{equation} 
where $\gamma>0$, $(s_k)_{k\ge 0}$ is the step size, and $T(s_k, L_{k+1})$ is an adaptive scheme to determine $s_{k+1}$ based on the previous step size $s_k$ and local smoothness parameter $L_{k+1}$. From \eqref{Algo: A-AGM}, we have  
\begin{equation}\label{E: inertial_iterate}
\begin{aligned}
t_{k+1}(x_{k+1}-y_{k+1}) 
&= (t_k-1)(y_{k+1}-y_k) + (\gamma-1)t_k( y_{k+1}-x_k ) \\
&= t_k(x_k-y_k) + \gamma t_k ( y_{k+1}-x_k ) - (y_{k+1}-y_k).  
\end{aligned}
\end{equation}

%%%%%%%%%%%%%%%%%%%%%%%%%%%%%%%%%%%%%%%%
%%%%%%%%%%%%%%%%%%%%%%%%%%%%%%%%%%%%%%%%
\subsection{The sequence \texorpdfstring{$(t_k)$}{tk} of inertial parameters}
The sequence $(t_k)_{k\ge 0}$ is defined by
\begin{equation}\label{E: t_k}
t_k = \left\{
\begin{array}{ccl}
t_0, &\qbox{if}& k=0, \\[3pt]
\frac{m+\sqrt{m^2 + 4t_{k-1}^2}}{2},&\qbox{if}& k\ge 1,
\end{array}
\right.
\end{equation}
where $m\in(0,1]$ and $t_0\ge 1$. It follows that $t_{k+1}^2 = t_k^2 + mt_{k+1}$ and 
$$ \frac{mk}{2}+t_0 \le t_k \le mk + t_0 ,\quad \forall k\ge 0. $$

\begin{remark}
The sequence $(t_k)_{k\ge 0}$ reduces to $(\theta_k)_{k\ge 0}$ in \eqref{E: theta_k} if $t_0=m=1$. As will be shown in Subsection \ref{Subsec: s_k}, letting $t_0\ge 1$ free, and $m\in(0,1)$ are two enabling components for an increasing step size. Also, with $m\in(0,1)$, we still have $t_k^2 = \mathcal{O}(k^2)$. 
\end{remark}

%%%%%%%%%%%%%%%%%%%%%%%%%%%%%%%%%%%%%%%%
%%%%%%%%%%%%%%%%%%%%%%%%%%%%%%%%%%%%%%%%
\subsection{The sequence \texorpdfstring{$(L_k)$ of local smoothness estimates}{Lk}}
We define
\begin{equation}\label{E: L_k}
L_{k+1}=\left\{
\begin{array}{ccl}
0, &\text{if}& \nabla f(x_{k+1}) = \nabla f(x_k),\\[4pt]
\frac{ \frac{1}{2}\| \nabla f(x_{k+1}) - \nabla f(x_k) \|^2 }{ \langle \nabla f(x_{k+1}), x_{k+1} - x_k \rangle - \left( f(x_{k+1}) - f(x_k) \right) }, &\text{if}& \nabla f(x_{k+1}) \neq \nabla f(x_k).  
\end{array}  
\right.
\end{equation} 
Observe that $L_{k+1}=0$ occurs only when $\nabla f(x_{k+1}) = \nabla f(x_k)$. Moreover, since $f$ is convex and $L$-smooth, we have
$$ \langle \nabla f(x_{k+1}), x_{k+1} - x_k \rangle - \left( f(x_{k+1}) - f(x_k) \right) \ge \frac{1}{2L}\| \nabla f(x_{k+1}) - \nabla f(x_k) \|^2. $$
This means, on the one hand, that $L_{k+1}\le L$, and, on the other, that $ \nabla f(x_{k+1}) = \nabla f(x_k) $ whenever $\langle \nabla f(x_{k+1}), x_{k+1} - x_k \rangle - \left( f(x_{k+1}) - f(x_k) \right) = 0$. Since no confusion should arise, we follow the convention $\frac{0}{0}=0$ and write \eqref{E: L_k} in a compact form as
\begin{equation}\label{E: L_k_bis}
L_{k+1} = \frac{ \frac{1}{2}\| \nabla f(x_{k+1}) - \nabla f(x_k) \|^2 }{ \langle \nabla f(x_{k+1}), x_{k+1} - x_k \rangle - \left( f(x_{k+1}) - f(x_k) \right) },
\end{equation}
for convenience. In a similar fashion, we set
$$ \frac{\| \nabla f(x_{k+1}) - \nabla f(x_k) \|^2}{L_{k+1}} = 0, $$
whenever $\nabla f(x_{k+1}) = \nabla f(x_k)$.

%%%%%%%%%%%%%%%%%%%%%%%%%%%%%%%%%%%%%%%%
%%%%%%%%%%%%%%%%%%%%%%%%%%%%%%%%%%%%%%%%
\subsection{The sequence \texorpdfstring{$(s_k)$ of step sizes}{sk}}\label{Subsec: s_k}

The step size $s_{k+1}$ is determined inductively as follows: pick $s_0>0$, $\omega\in[0,1)$, $\delta\in[0,1)$ and $\beta>0$. Given $s_k$, compute $s_{k+1}$ by:
\begin{equation} \label{E: s_k}
\left\{\begin{aligned}
    \qquad s_{k+1} & = \min\left\{  A_k s_k, B_k s_k, \frac{C_k}{L_{k+1}} \right\},\\
    \text{where}&\\
    A_k & = \frac{t_k^2}{t_{k+1}(t_{k+1}-1)} = \frac{t_{k+1}-m}{t_{k+1}-1}, \\[3pt]
    B_k & = \frac{2}{(1+\beta)\gamma}\left( 1 - \frac{1}{t_{k+1}} \right),\\[3pt]
    C_k & = \frac{1-\omega}{ \frac{(1+\beta)\gamma t_{k+1}}{t_{k+1}-1} + \frac{t_{k+1}(t_{k+1}-1)}{\beta(1-\delta)\gamma t_k^2} } = \frac{1-\omega}{ \frac{2}{B_k} + \frac{1}{ \beta(1-\delta)\gamma A_k } }.
\end{aligned}\right. 
\end{equation}
Some comments are in order:

First, since $m\in(0,1]$, we have
$$ A_k = \frac{t_k^2}{t_{k+1}(t_{k+1}-1)} = \frac{t_{k+1}-m}{t_{k+1}-1} \ge 1, $$
with strict inequality if $m\in(0,1)$. On the other hand,
$$ B_k = \frac{2}{(1+\beta)\gamma}\left( 1 - \frac{1}{t_{k+1}} \right) 
> \frac{2}{(1+\beta)\gamma}\left( 1 - \frac{1}{t_0} \right). $$ 
As a consequence, if the parameters $\beta,\gamma,t_0$ satisfy
\begin{equation} \label{E: B>1}
  \frac{2}{(1+\beta)\gamma}\left( 1 - \frac{1}{t_0} \right) \ge 1,
\end{equation}
then $B_k>1$. This means that an increasing step size $s_{k+1} > s_k$ is possible if $m\in(0,1)$ and \eqref{E: B>1} holds. Note that, for \eqref{E: B>1} to hold, it is necessary that $\gamma\in(0,2)$. 

Finally, since $A_k\ge 1$ and $B_k > \frac{2}{(1+\beta)\gamma}\left( 1 - \frac{1}{t_0} \right)$, we have
$$
\frac{1-\omega}{C_k}
= \frac{2}{B_k} + \frac{1}{ \beta\gamma(1-\delta) A_k } 
\le \frac{ (1+\beta)\gamma t_0}{t_0-1}
    + \frac{1}{\beta\gamma(1-\delta)},  
$$
so that
$$ \frac{C_k}{L_{k+1}} 
\ge \frac{1-\omega}{L_{k+1}\left( \frac{ (1+\beta)\gamma t_0 }{t_0-1} + \frac{1}{\beta\gamma(1-\delta)} \right)}
\ge \frac{1-\omega}{L\left( \frac{ (1+\beta)\gamma t_0 }{t_0-1} + \frac{1}{\beta\gamma(1-\delta)} \right)}. $$ 
This implies that $(s_k)_{k\ge 0}$ is bounded from below by $\frac{q}{L}$, where
\begin{equation} \label{E: kappa}
    q := \frac{1-\omega}{\frac{ (1+\beta)\gamma t_0 }{t_0-1}
    + \frac{1}{\beta\gamma(1-\delta)}},
\end{equation}
as shown in the following:

\begin{proposition}\label{Prop: s_k_lower_bound}
Let $m\in(0,1]$, let \eqref{E: B>1} hold, and let the sequence $(s_k)_{k\ge 0}$ be defined by \eqref{E: s_k}. If
$ s_0 \ge \frac{q}{L}$, then $s_k \ge \frac{q}{L}$ for every $k\ge 0$.
\end{proposition}

\begin{proof}
We prove the argument by contradiction. Let $s_{K+1}$ be the first term in the sequence $(s_k)_{k\ge 0}$ to satisfy $s_k<\frac{q}{L}$. This means that
$s_{K+1}<\frac{q}{L}\le s_K$. Since $A_K,B_K\ge 1$, we have
$$s_{K+1}=\min\left\{A_{K}s_{K},B_{K}s_{K},\frac{C_{K}}{L_{K+1}}\right\}\ge \min\left\{s_{K},\frac{C_{K}}{L_{K+1}}\right\} \ge \frac{q}{L},$$
which is impossible.
\end{proof}

On the other hand, the sequence $(s_k)_{k\ge 0}$ can grow at most as a power of $k$.

\begin{proposition} \label{Prop: s_k_upper_bound}
Let $m\in(0,1]$. Let $(s_k)_{k\ge 0}$ be defined by \eqref{E: s_k}. For every $k\ge 0$, we have
$$s_k\le \left[s_0e^{\frac{2(1-m)}{m}}\right]k^{\frac{2(1-m)}{m}}.$$
\end{proposition}

\begin{proof}
    By definition, 
    $$s_{k+1} \le A_ks_k = \left(\frac{t_{k+1}-m}{t_{k+1}-1}\right)s_k = \left(1+\frac{1-m}{t_{k+1}-1}\right)s_k$$ 
    for every $k\ge 0$. If $m=1$, the sequence is bounded from above. Otherwise, we write
    $$\ln(s_{k+1})-\ln(s_k)\le \ln\left(1+\frac{1-m}{t_{k+1}-1}\right) \le \frac{1-m}{t_{k+1}-1} \le \frac{2(1-m)}{m(k+1)}.$$
    Summing for $k=0,\dots,K-1$, we obtain
    $$\ln(s_K)-\ln(s_0)\le \frac{2(1-m)}{m}\sum_{k=1}^{K}\frac{1}{k}\le \frac{2(1-m)}{m}\left(1+\int_{1}^{K}\frac{d\zeta}{\zeta}\right)=\frac{2(1-m)}{m}\big(1+\ln(K)\big),$$
    which proves the result.     
\end{proof}

\begin{remark} 
A different bound on the growth of $s_k$ can be obtained using the inequality $s_{k+1}\le B_ks_k$.
\end{remark}

%%%%%%%%%%%%%%%%%%%%%%%%%%%%%%%%%%%%%%%%
%%%%%%%%%%%%%%%%%%%%%%%%%%%%%%%%%%%%%%%%
\subsection{The standing assumption on the parameters}

The discussion above motivates the following:

\begin{hypothesis} \label{Hypo: s_k}
Set $m\in(0,1)$, $\omega\in[0,1)$, $\delta\in[0,1)$, $\beta>0$ and $\gamma\in(0,2)$. Set $t_0\ge 1$ and $s_0\ge\frac{q}{L}$, where $q$ is given by \eqref{E: kappa}. Suppose that \eqref{E: B>1} holds. Let the sequences $(t_k)_{k\ge 0}$, $(L_k)_{k\ge 0}$ and $(s_k)_{k\ge 0}$ be given by \eqref{E: t_k}, \eqref{E: L_k_bis} and \eqref{E: s_k}, respectively.
\end{hypothesis}

\begin{remark} \label{Rem: ABC_k}
Under Hypothesis \ref{Hypo: s_k}, $A_k>1$, $B_k>1$ and $\frac{C_k}{L_{k+1}}\ge\frac{q}{L}$.
\end{remark}

%%%%%%%%%%%%%%%%%%%%%%%%%%%%%%%%%%%%%%%%
%%%%%%%%%%%%%%%%%%%%%%%%%%%%%%%%%%%%%%%%
\section{Energy estimations}\label{Sec: estimations}

Our convergence analysis centers around the energy sequence $(E_k)_{k\ge 0}$, given by
\begin{equation}\label{E: E_k}
E_k = E_k(x^*) := \frac{1}{2}\| \phi_k \|^2 + \frac{\beta}{2}\gamma^2 t_k^2 s_k^2 \| \nabla f(x_k) \|^2 + \gamma t_k^2 s_k \left( f(x_k) - f^* \right),
\end{equation}
where $\beta>0$, $x^*\in\argmin(f)$ (arbitrary, but fixed) and 
\begin{equation}\label{E: phi_k}
\begin{aligned}
\phi_k :&= t_{k+1}(x_{k+1}-y_{k+1}) + (y_{k+1}-x^*) \\
&= t_k(x_k-y_k) + \gamma t_k ( y_{k+1}-x_k ) + (y_k-x^*) \\
&= (t_k-1)(x_k-y_k) + \gamma t_k (y_{k+1}-x_k) + (x_k-x^*),
\end{aligned}
\end{equation}
in view of \eqref{E: inertial_iterate}.

We have the following:

\begin{lemma}\label{Lem: E_k_diff_bound}
Let $f:H\to\R$ be $\mu$-strongly convex and $L$-smooth. Let $(x_k)_{k\ge 0}$ and $(y_k)_{k\ge 0}$ be generated by \eqref{Algo: A-AGM}, and consider the sequence $(E_k)_{k\ge 0}$ defined by \eqref{E: E_k}. Then, for every $\omega\in[0,1)$, we have
\begin{align*}
E_{k+1}-E_k 
&\le \gamma\left[t_{k+1}(t_{k+1}-1) s_{k+1} - t_k^2 s_k \right] ( f(x_k) - f^* ) \\ 
&\quad + \frac{1+\beta}{2}\gamma^2 t_{k+1}^2 s_{k+1}^2 \| \nabla f(x_{k+1}) \|^2 
       - \frac{\beta}{2}\gamma^2 t_k^2 s_k^2 \| \nabla f(x_k) \|^2 \\
&\quad - \gamma t_{k+1}(t_{k+1}-1) s_{k+1}s_k \langle \nabla f(x_{k+1}), \nabla f(x_k) \rangle\\
&\quad - \frac{(1-\omega)\gamma t_{k+1} (t_{k+1}-1)  s_{k+1}}{2L_{k+1}}\| \nabla f(x_{k+1}) - \nabla f(x_k) \|^2 \\ 
&\quad - \frac{\omega\mu \gamma}{2} t_{k+1}(t_{k+1}-1)s_{k+1}\| x_{k+1}-x_k \|^2 
       - \frac{\mu\gamma}{2}t_{k+1}s_{k+1}\| x_{k+1} - x^* \|^2.
\end{align*}
\end{lemma}

\begin{proof}
By \eqref{E: E_k}, we have
\begin{equation}\label{E: E_k_diff}
\begin{aligned}
E_{k+1} - E_k
&= \left( \frac{1}{2}\| \phi_{k+1} \|^2 - \frac{1}{2}\| \phi_k \|^2 \right) \\ 
&\quad + \gamma\left[ t_{k+1}^2 s_{k+1} \left( f(x_{k+1})-f^* \right) - t_k^2 s_k \left( f(x_k) - f^* \right) \right] \\
&\quad + \beta\gamma^2\left( \frac{1}{2}t_{k+1}^2 s_{k+1}^2 \| \nabla f(x_{k+1}) \|^2 - \frac{1}{2} t_k^2 s_k^2 \| \nabla f(x_k) \|^2  \right).
\end{aligned}
\end{equation}
Using \eqref{E: phi_k} and \eqref{E: inertial_iterate}, we obtain
\begin{equation}\label{E: phi_diff} 
\begin{aligned}
\phi_{k+1}-\phi_k
&= t_{k+2}(x_{k+2}-y_{k+2}) - t_{k+1}(x_{k+1}-y_{k+1}) + (y_{k+2}-y_{k+1}) \\
&= \gamma t_{k+1}(y_{k+2}-x_{k+1}),
\end{aligned}
\end{equation}
which gives
$$ \| \phi_{k+1}-\phi_k \|^2 = \gamma^2 t_{k+1}^2 \| y_{k+2}-x_{k+1} \|^2, $$
and
\begin{align*}
&\quad \langle \phi_{k+1}-\phi_k, \phi_{k+1} \rangle \\
&= \gamma t_{k+1}\langle y_{k+2}-x_{k+1}, (t_{k+1}-1)(x_{k+1}-y_{k+1}) + \gamma t_{k+1}(y_{k+2}-x_{k+1}) + (x_{k+1}-x^*) \rangle \\
&= \gamma t_{k+1}(t_{k+1}-1) \langle y_{k+2}-x_{k+1}, x_{k+1}-y_{k+1} \rangle  
  + \gamma^2 t_{k+1}^2 \| y_{k+2}-x_{k+1} \|^2 \\ 
&\quad  + \gamma t_{k+1}\langle y_{k+2}-x_{k+1}, x_{k+1} - x^* \rangle. 
\end{align*}
Notice that in the first term,
$$ \langle y_{k+2}-x_{k+1}, x_{k+1}-y_{k+1} \rangle 
= \langle y_{k+2}-x_{k+1}, x_{k+1}-x_k \rangle
 + \langle y_{k+2}-x_{k+1}, x_k-y_{k+1} \rangle, $$
so that
\begin{align*}
\langle \phi_{k+1}-\phi_k, \phi_{k+1} \rangle 
&= \gamma t_{k+1}(t_{k+1}-1) \langle y_{k+2}-x_{k+1}, x_{k+1}-x_k \rangle\\ 
&\quad - \gamma t_{k+1} (t_{k+1}-1) \langle y_{k+2}-x_{k+1}, y_{k+1}-x_k \rangle \\
&\quad + \gamma^2 t_{k+1}^2 \| y_{k+2}-x_{k+1} \|^2 
  + \gamma t_{k+1}\langle y_{k+2}-x_{k+1}, x_{k+1} - x^* \rangle. 
\end{align*}
Since
$$ \frac{1}{2}\| \phi_{k+1} \|^2 - \frac{1}{2}\| \phi_k \|^2 = \langle \phi_{k+1}-\phi_k, \phi_{k+1} \rangle - \frac{1}{2}\| \phi_{k+1} - \phi_k \|^2, $$
we obtain
\begin{equation}\label{E: phi_diff_bound}
\begin{aligned}
\frac{1}{2}\| \phi_{k+1} \|^2 - \frac{1}{2}\| \phi_k \|^2  
&= \gamma t_{k+1}(t_{k+1}-1) \langle y_{k+2}-x_{k+1}, x_{k+1}-x_k \rangle \\  
&\quad - \gamma t_{k+1}(t_{k+1}-1) \langle y_{k+2}-x_{k+1}, y_{k+1}-x_k \rangle \\
&\quad + \frac{1}{2}\gamma^2 t_{k+1}^2 \| y_{k+2}-x_{k+1} \|^2 \\ 
&\quad + \gamma t_{k+1}\langle y_{k+2}-x_{k+1}, x_{k+1} - x^* \rangle \\
&= - \gamma t_{k+1}(t_{k+1}-1)s_{k+1} \langle \nabla f(x_{k+1}), x_{k+1}-x_k \rangle \\
&\quad - \gamma t_{k+1}(t_{k+1}-1)s_{k+1}s_k \langle \nabla f(x_{k+1}), \nabla f(x_k) \rangle \\
&\quad + \frac{1}{2}\gamma^2 t_{k+1}^2 s_{k+1}^2 \| \nabla f(x_{k+1}) \|^2 \\   
&\quad - \gamma t_{k+1}s_{k+1}\langle \nabla f(x_{k+1}), x_{k+1} - x^* \rangle, 
\end{aligned}
\end{equation}
by using $y_{k+2}= x_{k+1} - s_{k+1}\nabla f(x_{k+1})$ and $y_{k+1}=x_k - s_k\nabla f(x_k)$. The fact that $f$ is $\mu$-strongly convex gives
\begin{align}
\label{E: f_bound_1}
\langle \nabla f(x_{k+1}), x_{k+1}-x^* \rangle &\ge f(x_{k+1}) - f^* + \frac{\mu}{2}\| x_{k+1} - x^* \|^2,\\
\label{E: f_bound_2}
\langle \nabla f(x_{k+1}), x_{k+1} - x_k \rangle &\ge f(x_{k+1}) - f(x_k) + \frac{\mu}{2}\| x_{k+1} - x_k \|^2. 
\end{align} 
Keeping in mind that $\omega\in[0,1)$ and
\begin{equation}\label{E: f_bound_3}
\langle \nabla f(x_{k+1}), x_{k+1} - x_k \rangle = f(x_{k+1}) - f(x_k) + \frac{1}{2L_{k+1}}\| \nabla f(x_{k+1}) - \nabla f(x_k) \|^2,
\end{equation}
we combine \eqref{E: f_bound_2} and \eqref{E: f_bound_3} to obtain
\begin{align*}
\langle \nabla f(x_{k+1}), x_{k+1} - x_k \rangle 
&\ge f(x_{k+1}) - f(x_k) + \frac{1-\omega}{2L_{k+1}}\| \nabla f(x_{k+1}) - \nabla f(x_k) \|^2 \\
&\quad + \frac{\omega\mu}{2}\| x_{k+1} - x_k \|^2.
\end{align*}
Using this inequality, together with \eqref{E: f_bound_1}, in \eqref{E: phi_diff_bound} gives
\begin{align*}
&\quad \frac{1}{2}\| \phi_{k+1} \|^2 - \frac{1}{2}\| \phi_k \|^2 \\
&\le - \gamma t_{k+1}(t_{k+1}-1) s_{k+1}( f(x_{k+1}) - f(x_k) ) 
     - \gamma t_{k+1}s_{k+1}( f(x_{k+1}) - f^* ) \\
&\quad + \frac{1}{2}\gamma^2 t_{k+1}^2 s_{k+1}^2 \| \nabla f(x_{k+1}) \|^2 
     - \gamma t_{k+1}(t_{k+1}-1) s_{k+1}s_k \langle \nabla f(x_{k+1}), \nabla f(x_k) \rangle\\
&\quad - \frac{(1-\omega)\gamma t_{k+1} (t_{k+1}-1)  s_{k+1}}{2L_{k+1}}\| \nabla f(x_{k+1}) - \nabla f(x_k) \|^2 \\ 
&\quad - \frac{\omega\mu \gamma}{2} t_{k+1}(t_{k+1}-1)s_{k+1}\| x_{k+1}-x_k \|^2 
       - \frac{\mu\gamma}{2}t_{k+1}s_{k+1}\| x_{k+1} - x^* \|^2.
\end{align*}
Using this inequality in \eqref{E: E_k_diff} gives the desired result.
\end{proof}

We proceed to present a crucial result that will enable our adaptive accelerated gradient method possible.

\begin{lemma}\label{Lem: E_k_diff_bound_SC}
Let $f:H\to\R$ be $\mu$-strongly convex and $L$-smooth. Let $(x_k)_{k\ge 0}$ and $(y_k)_{k\ge 0}$ be generated by \eqref{Algo: A-AGM}, with $\gamma\in(0,2)$, and consider the sequence $(E_k)_{k\ge 0}$ defined by \eqref{E: E_k}. If Hypothesis \ref{Hypo: s_k} holds, then
\begin{align*}
E_{k+1}-E_k
&\le - \frac{\omega\mu\gamma t_{k+1} (t_{k+1}-1)  s_{k+1}}{2}\| x_{k+1}-x_k \|^2 
 - \frac{\mu \gamma t_{k+1}s_{k+1}}{2}\| x_{k+1}-x^* \|^2 \\ 
&\quad - \frac{\beta\delta\gamma^2 t_k^2 s_k^2}{2} \| \nabla f(x_k) \|^2.
\end{align*} 
\end{lemma}

\begin{proof}
By Lemma \ref{Lem: E_k_diff_bound}, we have
\begin{align*}
E_{k+1}-E_k
&\le - \frac{\omega\mu\gamma t_{k+1} (t_{k+1}-1)  s_{k+1}}{2}\| x_{k+1}-x_k \|^2 
 - \frac{\mu \gamma t_{k+1}s_{k+1}}{2}\| x_{k+1}-x^* \|^2 \\ 
&\quad - \frac{\beta\delta\gamma^2 t_k^2 s_k^2}{2} \| \nabla f(x_k) \|^2 + R_k,
\end{align*}
where
\begin{align*}
R_k &= \gamma\left[t_{k+1}(t_{k+1}-1) s_{k+1} - t_k^2 s_k \right] ( f(x_k) - f^* ) \\ 
&\quad + \frac{1+\beta}{2}\gamma^2 t_{k+1}^2 s_{k+1}^2  \| \nabla f(x_{k+1}) \|^2  
 - \frac{\beta(1-\delta)}{2}\gamma^2 t_k^2 s_k^2 \| \nabla f(x_k) \|^2 \\ 
&\quad - \gamma t_{k+1}(t_{k+1}-1)s_{k+1}s_k \langle \nabla f(x_{k+1}), \nabla f(x_k) \rangle \\
&\quad - \frac{(1-\omega)\gamma t_{k+1}(t_{k+1}-1) s_{k+1}}{2L_{k+1}}\| \nabla f(x_{k+1}) - \nabla f(x_k) \|^2.
\end{align*}
If we can show $R_k\le 0$, then the proof is done. Since $s_{k+1} \le A_k s_k$, we have
\begin{align*}
R_k &\le \frac{1+\beta}{2}\gamma^2 t_{k+1}^2 s_{k+1}^2  \| \nabla f(x_{k+1}) \|^2  
 - \frac{\beta(1-\delta)}{2}\gamma^2 t_k^2 s_k^2 \| \nabla f(x_k) \|^2 \\ 
&\quad - \gamma t_{k+1}(t_{k+1}-1)s_{k+1}s_k \langle \nabla f(x_{k+1}), \nabla f(x_k) \rangle \\
&\quad - \frac{(1-\omega)\gamma t_{k+1}(t_{k+1}-1) s_{k+1}}{2L_{k+1}}\| \nabla f(x_{k+1}) - \nabla f(x_k) \|^2.
\end{align*}
Suppose first that $\nabla f(x_{k+1}) = \nabla f(x_k)$. Then,
\begin{align*}
R_k 
&\le \gamma t_{k+1}s_{k+1} \left[ \frac{1+\beta}{2} \gamma t_{k+1} s_{k+1} - (t_{k+1}-1)s_k \right] \| \nabla f(x_{k+1}) \|^2 \\
&\quad - \frac{\beta(1-\delta)}{2}\gamma^2 t_k^2 s_k^2 \| \nabla f(x_k) \|^2,
\end{align*}
which implies that $R_k \le - \frac{\beta(1-\delta)}{2}\gamma^2 t_k^2 s_k^2 \| \nabla f(x_k) \|^2 \le 0$ since $s_{k+1} \le B_k s_k$.

If, on the other hand, $\nabla f(x_{k+1}) \ne \nabla f(x_k)$, then
\begin{align*}
R_k
&\le -\frac{(1-\omega)\gamma t_{k+1}(t_{k+1}-1)s_{k+1}}{2L_{k+1}}\left( 1 - \frac{(1+\beta)\gamma t_{k+1} s_{k+1} L_{k+1} }{ (1-\omega)(t_{k+1}-1) }  \right) \| \nabla f(x_{k+1}) \|^2 \\
&\quad -\frac{(1-\omega)\gamma t_{k+1}(t_{k+1}-1)s_{k+1}}{2L_{k+1}}\left( 1 + \frac{\beta(1-\delta)\gamma t_k^2 s_k^2 L_{k+1}}{(1-\omega)t_{k+1}(t_{k+1}-1)s_{k+1}} \right) \| \nabla f(x_k) \|^2 \\
&\quad -\frac{(1-\omega)\gamma t_{k+1}(t_{k+1}-1)s_{k+1}}{L_{k+1}}\left( \frac{s_k L_{k+1}}{1-\omega} - 1 \right)\langle \nabla f(x_{k+1}), \nabla f(x_k) \rangle\\
&= -\frac{(1-\omega)\gamma t_{k+1}(t_{k+1}-1)s_{k+1}}{2L_{k+1}} W_k,
\end{align*}
where
\begin{align*}
W_k &= \left( 1 - \frac{(1+\beta)\gamma t_{k+1} s_{k+1} L_{k+1} }{ (1-\omega)(t_{k+1}-1) }  \right) \| \nabla f(x_{k+1}) \|^2 \\
&\quad + \left( 1 + \frac{\beta(1-\delta)\gamma t_k^2 s_k^2 L_{k+1}}{(1-\omega)t_{k+1}(t_{k+1}-1)s_{k+1}} \right) \| \nabla f(x_k) \|^2 \\ 
&\quad + 2\left( \frac{s_k L_{k+1}}{1-\omega} - 1 \right)\langle \nabla f(x_{k+1}), \nabla f(x_k) \rangle.
\end{align*}
The discriminant of the quadratic form $\frac{1}{2}W_k$ is
\begin{align*}
\Delta_k 
&= \left( \frac{s_k L_{k+1}}{1-\omega} - 1 \right)^2 \\
&\quad  - \left( 1 - \frac{(1+\beta)\gamma t_{k+1} s_{k+1} L_{k+1} }{ (1-\omega)(t_{k+1}-1) }  \right)
   \left( 1 + \frac{\beta(1-\delta)\gamma t_k^2 s_k^2 L_{k+1}}{(1-\omega)t_{k+1}(t_{k+1}-1)s_{k+1}} \right) \\
&= \left[ 1-\left( 1 - \frac{(1+\beta)\gamma t_{k+1} s_{k+1} L_{k+1} }{ (1-\omega)(t_{k+1}-1) }  \right) \frac{\beta(1-\omega)(1-\delta)\gamma t_k^2 }{t_{k+1}(t_{k+1}-1)s_{k+1}L_{k+1} }  \right] \left( \frac{ s_k L_{k+1} }{1-\omega} \right)^2 \\ 
&\quad + \frac{2L_{k+1}}{1-\omega}\left[\frac{(1+\beta)\gamma t_{k+1} }{ 2(t_{k+1}-1) } s_{k+1} - s_k \right] \\
& \le \left[ 1-\left( 1 - \frac{(1+\beta)\gamma t_{k+1} s_{k+1} L_{k+1} }{ (1-\omega)(t_{k+1}-1) }  \right) \frac{\beta(1-\omega)(1-\delta)\gamma t_k^2 }{t_{k+1}(t_{k+1}-1)s_{k+1}L_{k+1} }  \right] \left( \frac{ s_k L_{k+1} }{1-\omega} \right)^2,
\end{align*}
because $s_{k+1}\le B_ks_k$. On the other hand,
\begin{align*}
&\qquad \left( 1 - \frac{(1+\beta)\gamma t_{k+1} s_{k+1} L_{k+1} }{ (1-\omega)(t_{k+1}-1) }  \right) \frac{\beta(1-\omega)(1-\delta)\gamma t_k^2 }{t_{k+1}(t_{k+1}-1)s_{k+1}L_{k+1} } \ge 1 \\
&\Longleftrightarrow \left[ \frac{(1+\beta)\gamma t_{k+1} }{ (1-\omega)(t_{k+1}-1) } + \frac{t_{k+1}(t_{k+1}-1) }{\beta(1-\omega)(1-\delta)\gamma t_k^2 } \right]s_{k+1} L_{k+1} \le 1 \\
&\Longleftrightarrow \frac{s_{k+1} L_{k+1}}{C_k} \le 1,
\end{align*}
which is is true by the definition of $s_{k+1}$. Therefore, we have $\Delta_k\le 0$, which gives $W_k \ge 0$, whence $R_k\le 0$. This completes the proof.
\end{proof}

%%%%%%%%%%%%%%%%%%%%%%%%%%%%%%%%%%%%%%%%
%%%%%%%%%%%%%%%%%%%%%%%%%%%%%%%%%%%%%%%%
\section{Convergence analysis I: the convex case}\label{Sec: convexity}
In this section, we establish a fast convergence rate of the function values, and prove the weak convergence of the iterates for \eqref{Algo: A-AGM}, when $f$ is convex.

%%%%%%%%%%%%%%%%%%%%%%%%%%%%%%%%%%%%%%%%
%%%%%%%%%%%%%%%%%%%%%%%%%%%%%%%%%%%%%%%%
\subsection{Convergence of the function values}\label{Sub: function_values}

In this subsection, we prove the convergence rate of the function values.

\begin{theorem}\label{Thm: func_values}
Let $f:H\to\R$ be convex and $L$-smooth. Let $(x_k)_{k\ge 0}$ and $(y_k)_{k\ge 0}$ be generated by \eqref{Algo: A-AGM}, with $\gamma\in(0,2)$, and assume Hypothesis \ref{Hypo: s_k} to hold with $\omega=\delta=0$. Given $x_0=y_0$, for every $k\ge 0$, we have
$$ f(x_k) - f^* \le \frac{DL}{t_k^2}, $$ 
where
$$D=\frac{1}{q}\left[ \tfrac{1}{2\gamma}\| x_0 - x^* \|^2 + \tfrac{s_0t_0\big( (1+\beta)\gamma s_0t_0L - 1 \big)}{2L}   \| \nabla f(x_0) \|^2 + s_0t_0(t_0-1) ( f(x_0) - f^* ) \right].$$
In particular, every weak subsequential limit point of $x_k$, as $k\to\infty$, minimizes $f$.
\end{theorem}

\begin{proof}
Setting $\mu=0$, $\omega=0$ and $\delta=0$ in Lemma \ref{Lem: E_k_diff_bound_SC}, we obtain $E_{k+1} - E_k \le 0$, which gives
$$ f(x_k) - f^* \le \frac{E_k}{\gamma t_k^2 s_k} \le \frac{E_0}{\gamma t_k^2 s_k}. $$
With \eqref{E: phi_k} and $x_0 = y_0$ in mind, we have
\begin{equation}\label{E: E_0_bound}
\begin{aligned}
E_0 
&= \frac{1}{2}\| -\gamma s_0t_0\nabla f(x_0) + (x_0-x^*) \|^2 + \frac{\beta}{2}\gamma^2 t_0^2 s_0^2 \| \nabla f(x_0) \|^2 \\
&\quad + \gamma t_0^2 s_0( f(x_0) - f^* ) \\
&= \frac{1}{2}\| x_0 - x^* \|^2 + \frac{1+\beta}{2}\gamma^2 t_0^2 s_0^2 \| \nabla f(x_0) \|^2 
   + \gamma t_0(t_0-1)s_0( f(x_0) - f^* )\\
&\quad - \gamma s_0t_0 \left( \langle \nabla f(x_0), x_0 - x^* \rangle - ( f(x_0) - f^* ) \right) \\
&\le \frac{1}{2}\| x_0 - x^* \|^2 + \frac{1}{2}\gamma s_0t_0 \left[ (1+\beta)\gamma s_0t_0 - \frac{1}{L} \right] \| \nabla f(x_0) \|^2 \\
&\quad + \gamma t_0(t_0-1)s_0 ( f(x_0) - f^* ).
\end{aligned}
\end{equation}
From Proposition \ref{Prop: s_k_lower_bound}, we obtain $s_k \ge \frac{q}{L}$. As a result,
$$ f(x_k) - f^* \le \frac{E_0}{\gamma t_k^2 s_k} \le \frac{E_0L}{q\gamma t_k^2}\le \frac{DL}{t_k^2}, $$
which allows us to conclude. For the minimizing property, it suffices to observe that $f$ is weakly lower-semicontinuous, and that $t_k\to\infty$ as $k\to\infty$.
\end{proof}

\begin{remark} \label{Rem: bound D}
Since $\|\nabla f(x)\|^2\le 2L(f(x)-f^*)$ and $\|\nabla f(x)\|^2\le L^2\|x-x^*\|^2$, we can bound
\begin{align*}
D & \le \frac{1}{q}\min\left\{\frac{1}{2\gamma}\| x_0 - x^* \|^2 + s_0t_0\left[ t_0\big((1+\beta)\gamma s_0L+1\big) - 2\right]  ( f(x_0) - f^* ),\right. \\
& \qquad \left. \frac{1}{2\gamma}\Big[1 + \gamma s_0t_0L \big( (1+\beta)\gamma s_0t_0L - 1\big)\Big]\| x_0 - x^* \|^2  + s_0t_0(t_0-1) ( f(x_0) - f^* )\right\}.
\end{align*}    
\end{remark}

For the reader's convenience, we give explicit convergence rates for two particular choices of the parameters:

{\bf The case  $\gamma=\frac{1}{2}$, $\beta=1$ and $t_0=2$.} Theorem \ref{Thm: func_values} gives:  

\begin{corollary}\label{Cor: rate_cvx_1}
Let $f:H\to\R$ be convex and $L$-smooth. Let the sequence $(t_k)_{k\ge 0}$ be given by \eqref{E: t_k} with $m\in(0,1)$ and $t_0=2$. Consider the algorithm \eqref{Algo: A-AGM}, where $\gamma = \frac{1}{2}$ and Hypothesis \ref{Hypo: s_k} holds with $\omega=\delta=0$ and $\beta=1$. 
Then, given $s_0\ge \frac{1}{4L}$ and $x_0=y_0$, we have, for every $k\ge 0$,
$$ s_k\ge \frac{1}{4L},\qbox{and} f(x_k) - f^* \le \frac{4DL}{t_k^2}, $$
with $D= \| x_0 - x^* \|^2 + s_0\left( 2s_0 - \frac{1}{L} \right)\| \nabla f(x_0) \|^2 + 2 s_0 ( f(x_0) - f^* )$.
\end{corollary}

{\bf The case $\gamma=1$, $\beta=\frac{1}{3}$ and $t_0=3$.} Now Theorem \ref{Thm: func_values} becomes:

\begin{corollary}\label{Cor: rate_cvx_2}
Let $f:H\to\R$ be convex and $L$-smooth. Let the sequence $(t_k)_{k\ge 0}$ be given by \eqref{E: t_k} with $m\in(0,1)$ and $t_0=3$. Consider the algorithm \eqref{Algo: A-AGM}, where $\gamma = 1$ and Hypothesis \ref{Hypo: s_k} holds with $\omega=\delta=0$ and $\beta=\frac{1}{3}$. Then, given $s_0\ge \frac{1}{5L}$ and $x_0=y_0$, we have, for every $k\ge 0$,
$$ s_k\ge \frac{1}{5L},\qbox{and} f(x_k) - f^* \le \frac{5DL}{2t_k^2}, $$
with $D= \| x_0 - x^* \|^2 + 3s_0\left( 4s_0 - \frac{1}{L} \right)\| \nabla f(x_0) \|^2 + 12 s_0 ( f(x_0) - f^* )$.
\end{corollary}

\begin{remark}
Setting $\gamma=1$, \eqref{Algo: A-AGM} reduces to 
\begin{equation*}
\left\{
\begin{array}{rcl}
y_{k+1} &=& x_k - s_k\nabla f(x_k),\\[4pt]
x_{k+1} &=& y_{k+1} + \frac{t_k - 1}{t_{k+1}}( y_{k+1} - y_k ),
\end{array} 
\right.
\end{equation*} 
which takes the form of Nesterov's method \cite{Nesterov_1983}.
\end{remark}

%%%%%%%%%%%%%%%%%%%%%%%%%%%%%%%%%%%%%%%%
%%%%%%%%%%%%%%%%%%%%%%%%%%%%%%%%%%%%%%%%
\subsection{Convergence of the iterates}\label{Sub: iterates}

In this subsection, we prove the weak convergence of the iterates generated by \eqref{Algo: A-AGM}, following \cite{Ryu_2025}. 

\begin{lemma}\label{Lem: bounded_sequence}
Let $f:H\to\R$ be convex and $L$-smooth. Let $(x_k)_{k\ge 0}$ and $(y_k)_{k\ge 0}$ be generated by \eqref{Algo: A-AGM}, with $\gamma\in(0,2)$. If Hypothesis \ref{Hypo: s_k} holds, then
$(x_k)_{k\ge 0}$ and $(y_k)_{k\ge 0}$ are bounded.
\end{lemma}

\begin{proof}
Recall that the sequence $(E_k)_{k\ge 0}$ is given by \eqref{E: E_k}, where $x^*$ is an arbitrary minimizer of $f$. By Lemma \ref{Lem: E_k_diff_bound_SC}, $E_{k+1}\le E_k$, and $\lim\limits_{k\to\infty}E_k$ exists.
It follows that
$$\|\phi_k\|^2\le 2E_k\le 2E_0.$$
Writing
$$z_{k+1} = t_{k+1}(x_{k+1}-y_{k+1}) + y_{k+1},$$
we obtain 
$$\| z_{k+1}-x^* \|^2 \le 2E_0,$$
so that $z_k$ is bounded, say $\| z_k \| \le M$ for some constant $M>0$. By \eqref{E: phi_diff}, we have
\begin{equation} \label{E: z_k_diff}
    z_{k+1}-z_k= \gamma t_{k}(y_{k+1}-x_{k}),
\end{equation}
and so
$$\gamma t_{k} \| y_{k+1}-x_{k} \|=\| z_{k+1}-z_k \| \le 2M,$$
which gives
$$\|y_{k+1}\|\le\|x_k\|+\frac{2M}{\gamma t_k}.$$
Since $z_k = t_{k}(x_{k}-y_{k}) + y_{k}$, we deduce that
\begin{equation} \label{E: y to x}
    \|x_{k}\| \le \left(1-\frac{1}{t_{k}}\right)\|y_{k}\|+\frac{M}{t_{k}},
\end{equation}
whence
\begin{align*}
    \|y_{k+1}\| & \le \left(1-\frac{1}{t_{k}}\right)\|y_{k}\|+\frac{M}{t_{k}}\left(1+\frac{2}{\gamma}\right) \\
    & \le \max\left\{\|y_k\|,\left(1+\frac{2}{\gamma}\right)M\right\} \\
    & \le \max\left\{\|y_0\|,\left(1+\frac{2}{\gamma}\right)M\right\}.
\end{align*}
It follows that $(y_k)_{k\ge 0}$ is bounded and, by \eqref{E: y to x}, so is $(x_k)_{k\ge 0}$.
\end{proof}

The following result \cite[Lemma A.4]{Radu_2025} will be useful in the sequel:

\begin{lemma}\label{Lem: convergence_limit}
Let $(u_k)_{k\ge 0}$ be a real sequence and $(\zeta_k)_{k\ge 0}$ be positive such that $\sum_{k=0}^{\infty}\frac{1}{\zeta_k}=\infty$. If
$ \lim_{k\to\infty}[ u_{k+1} + \zeta_k( u_{k+1} - u_k ) ] = b\in\R, $
then,
$ \lim_{k\to\infty}u_k = b. $ 
\end{lemma}

We are now ready to prove the weak convergence of the iterates of \eqref{Algo: A-AGM}.

\begin{theorem}
Let $f:H\to\R$ be convex and $L$-smooth. Let $(x_k)_{k\ge 0}$ and $(y_k)_{k\ge 0}$ be generated by \eqref{Algo: A-AGM}, with $\gamma\in(0,2)$. If Hypothesis \ref{Hypo: s_k} holds, then $x_k$ and $y_k$ converge weakly, as $k\to\infty$, to the same point in $\argmin(f)$.
\end{theorem}

\begin{proof}
By Theorem \ref{Thm: func_values}, every weak subsequential limit point of $x_k$, as $k\to\infty$, is a minimizer of $f$. Since, by Lemma \ref{Lem: bounded_sequence}, $(x_k)_{k\ge 0}$ is bounded, it suffices to prove that it cannot have more than one such limit point. Suppose, then, that $x_{m_k} \rightharpoonup z^*$ and $x_{n_k}\rightharpoonup \tilde{z}^*$, as $k\to\infty$. Then $z^*$ and $\tilde z^*$ belong to $\argmin(f)$. Now, set
$$ U_k := E_k(z^*) - E_k(\tilde{z}^*), $$
and recall $z_{k+1} = t_{k+1}(x_{k+1}-y_{k+1}) + y_{k+1}$. It follows from Lemma \ref{Lem: E_k_diff_bound_SC} that $\lim_{k\to\infty} U_k$ exists. Using \eqref{E: E_k}, we obtain
\begin{align*}
U_k &= \frac{1}{2}\| z_{k+1} - z^* \|^2 - \frac{1}{2}\| z_{k+1} - \tilde{z}^* \|^2 \\
&= -\langle z_{k+1}, z^* - \tilde{z}^* \rangle + \frac{1}{2}\| z^* \|^2 - \frac{1}{2}\| \tilde{z}^* \|^2 \\
&= -\langle t_{k+1} x_{k+1} - (t_{k+1}-1)y_{k+1}, z^* - \tilde{z}^* \rangle + \frac{1}{2}\| z^* \|^2 - \frac{1}{2}\| \tilde{z}^* \|^2, 
\end{align*}
so that
$$ -2 \langle t_{k+1} x_{k+1} - (t_{k+1}-1)y_{k+1}, z^* - \tilde{z}^* \rangle + \| z^* \|^2 - \| \tilde{z}^* \|^2 = 2U_k. $$
Define
\begin{align*}
u_k &= \| x_{k+1} - z^* \|^2 - \| x_{k+1} - \tilde{z}^* \|^2= -2\langle x_{k+1}, z^* - \tilde{z}^* \rangle + \| z^* \|^2 - \| \tilde{z}^* \|^2,\\
\xi_k &= \| y_{k+1} - z^* \|^2 - \| y_{k+1} - \tilde{z}^* \|^2 = -2\langle y_{k+1}, z^* - \tilde{z}^* \rangle + \| z^* \|^2 - \| \tilde{z}^* \|^2,
\end{align*}
so that
\begin{equation} \label{E: U_k}
    t_{k+1}u_k - (t_{k+1}-1)\xi_k = 2U_k.
\end{equation}
Using \eqref{E: z_k_diff}, we have $z_{k+1}-z_k =  \gamma t_{k}(y_{k+1}-x_{k})$, which gives
\begin{equation}\label{E: y_k_eqn}
y_{k+1}=x_k+\frac{1}{\gamma t_k}(z_{k+1}-z_k).
\end{equation}
%$$z_{k+1}-z_k =  \gamma t_{k}(y_{k+1}-x_{k}),$$
With this, we can rewrite $\xi_k$ as
\begin{align*}
\xi_k
&=-2\left\langle x_k + \frac{z_{k+1}-z_k}{\gamma t_k}, z^* - \tilde{z}^* \right\rangle 
  + \| z^* \|^2 - \| \tilde{z}^* \|^2 \\
&= u_{k-1} - \frac{2}{\gamma t_k} \langle z_{k+1}-z_k, z^* - \tilde{z}^* \rangle \\
&= u_{k-1} + \frac{2}{\gamma t_k}( U_k - U_{k-1} ).
\end{align*}
Substituting this equality into \eqref{E: U_k} gives
$$ u_k + (t_{k+1}-1)(u_k-u_{k-1}) = 2U_k + \frac{2(t_{k+1}-1)}{\gamma t_k}( U_k - U_{k-1} ). $$
Since $\frac{t_{k+1}-1}{t_k} < \frac{t_k - (1-m)}{t_k} < 1$, and $\lim_{k\to\infty} U_k$ exists, we deduce that the right-hand side has a limit. Invoking Lemma \ref{Lem: convergence_limit}, we deduce that $\lim_{k\to\infty}u_k$ exists. Replacing $k$ by $m_k$, and then by $n_k$ in the definition of $u_k$, and taking limits, we obtain
$$ -\|z^* - \tilde{z}^* \| = \lim_{k\to\infty}u_k= \| \tilde{z}^* - z^* \|^2, $$
which implies that $z^* = \tilde{z}^*$. It ensues that $x_k$ converges weakly to a minimizer of $f$. From \eqref{E: y_k_eqn}, it easily follows that $y_k$ converges weakly to the same limit.
\end{proof}

%%%%%%%%%%%%%%%%%%%%%%%%%%%%%%%%%%%%%%%%
%%%%%%%%%%%%%%%%%%%%%%%%%%%%%%%%%%%%%%%%
\section{Convergence analysis II: the strongly convex case}\label{Sec: strong convexity}
In this section, we establish a linear convergence rate of the function values when $f$ is strongly convex.

\begin{lemma}\label{Lem: E_k_diff_SC}
Let $f:H\to\R$ be $\mu$-strongly convex and $L$-smooth, where $L\ge\mu>0$. Let $(x_k)_{k\ge 0}$ and $(y_k)_{k\ge 0}$ be generated by \eqref{Algo: A-AGM}, with $\gamma\in(0,2)$. Let Hypothesis \ref{Hypo: s_k} hold with $\omega=\delta=\frac{1}{2}$. Consider the sequence $(E_k)_{k\ge 0}$ defined by \eqref{E: E_k}. Then,
\begin{align*}
E_{k+1}-E_k
&\le - \frac{\mu\gamma t_{k+1} (t_{k+1}-1)  s_{k+1}}{4}\| x_{k+1}-x_k \|^2 
 - \frac{\mu \gamma t_{k+1}s_{k+1}}{2}\| x_{k+1}-x^* \|^2 \\ 
&\quad - \frac{\beta\gamma^2 t_k^2 s_k^2}{4} \| \nabla f(x_k) \|^2. 
\end{align*}
\end{lemma}

\begin{proof}
The result follows by setting $\omega=\frac{1}{2}$ and $\delta=\frac{1}{2}$ in Lemma \ref{Lem: E_k_diff_bound_SC}.
\end{proof}

In what follows, we give an upper bound for $(E_k)_{k\ge 0}$.

\begin{lemma}\label{Lem: E_k_bound_SC}
Let $f:H\to\R$ be $\mu$-strongly convex and $L$-smooth, where $L\ge\mu>0$. Let $(E_k)_{k\ge 0}$ be defined by \eqref{E: E_k}. Then,
\begin{align*}
E_k 
&\le \frac{(t_{k+1}-1)^2 (1+\eta+\sigma) }{2} \| x_{k+1}-x_k \|^2 
   + \frac{1}{2}\left( 1 + \frac{1}{\sigma} + \lambda \right) \| x_{k+1}-x^* \|^2 \\
&\quad + \frac{s_k^2}{2} \left[ \left( 1 + \frac{1}{\eta} + \frac{1}{\lambda} \right)(t_{k+1}-1)^2 + \beta\gamma^2 t_k^2 + \frac{\gamma t_k^2}{\mu s_k} \right]  \| \nabla f(x_k) \|^2,
\end{align*}
for every $\eta,\sigma,\lambda>0$.
\end{lemma}

\begin{proof}
We begin with \eqref{E: phi_k}, to obtain
\begin{align*}
\phi_k 
&= t_{k+1}(x_{k+1}-y_{k+1}) + (y_{k+1}-x^*) \\
&= (t_{k+1}-1)(x_{k+1}-y_{k+1}) + (x_{k+1}-x^*)\\
&= (t_{k+1}-1)(x_{k+1}-x_k) + (t_{k+1}-1)s_k\nabla f(x_k) + (x_{k+1}-x^*), 
\end{align*}
so that
\begin{align*}
\frac{1}{2}\| \phi_k \|^2
&\le \frac{(t_{k+1}-1)^2 (1+\eta+\sigma) }{2} \| x_{k+1}-x_k \|^2 
     + \frac{1}{2}\left( 1 + \frac{1}{\sigma} + \lambda \right) \| x_{k+1}-x^* \|^2\\
&\quad + \frac{ (t_{k+1}-1)^2 }{2}\left( 1 + \frac{1}{\eta} + \frac{1}{\lambda} \right) s_k^2 \| \nabla f(x_k) \|^2,
\end{align*}
for every $\eta,\sigma,\lambda>0$. Since $f$ is $\mu$-strongly convex, we have
$$ f(x_k) - f^* \le \frac{\| \nabla f(x_k) \|^2}{2\mu}. $$
Using the last two inequalities in \eqref{E: E_k}, we arrive at the desired result.
\end{proof}

We are now in a position to derive the convergence results under strong convexity.

\begin{theorem}\label{Thm: func_value_SC}
Let $f:H\to\R$ be $\mu$-strongly convex and $L$-smooth, where $L\ge\mu>0$. Let $(x_k)_{k\ge 0}$ and $(y_k)_{k\ge 0}$ be generated by \eqref{Algo: A-AGM}, with $\gamma\in(0,2)$. Let Hypothesis \ref{Hypo: s_k} hold with $\omega=\delta=\frac{1}{2}$. Given $x_0=y_0$, for every $k\ge 0$, we have
$$ f(x_k) - f^* \le \frac{D L}{t_k^2}(1-\rho)^k, $$ 
where
\begin{align*}
\rho &= \min\left\{ \frac{\mu \gamma q}{4L}, \frac{\mu q}{ \frac{2L}{\beta\gamma} + \left( \frac{8}{\beta\gamma^2} + 2 \right)\mu q } \right\},\\
D &= \frac{1}{q}\left[ \tfrac{1}{2\gamma}\| x_0 - x^* \|^2 + \tfrac{s_0t_0\left( (1+\beta)\gamma s_0t_0L - 1 \right)}{2L}   \| \nabla f(x_0) \|^2 + t_0(t_0-1)s_0 ( f(x_0) - f^* ) \right]. 
\end{align*}
\end{theorem}

\begin{proof}
Using Lemmas \ref{Lem: E_k_diff_SC} and \ref{Lem: E_k_bound_SC}, we obtain
$$ E_{k+1}-E_k \le -r E_k, $$
where
$$ r =\max_{\eta,\sigma,\lambda>0}\min\left\{ \frac{\mu\gamma s_{k+1} t_{k+1} }{2(t_{k+1}-1) (1+\eta+\sigma)}, 
\frac{\mu \gamma s_{k+1} t_{k+1}}{1 + \frac{1}{\sigma} + \lambda },
\frac{\mu s_k}{ 2 \left( \frac{1}{\beta\gamma} +  \varphi_k \mu s_k \right) }
\right\}, $$
with
$$ \varphi_k = \frac{ 1 + \frac{1}{\eta} + \frac{1}{\lambda} }{ \beta\gamma^2 } \left( \frac{t_{k+1}-1}{t_k} \right)^2 + 1 
\le \frac{ 1 + \frac{1}{\eta} + \frac{1}{\lambda} }{ \beta\gamma^2 } + 1. $$
Setting $\eta=\sigma=\frac{1}{2}$ and $\lambda=1$, we obtain
$$ r\ge \min\left\{ \frac{\mu \gamma s_{k+1}}{4}, \frac{\mu s_k}{ \frac{2}{\beta\gamma} + \left( \frac{8}{\beta\gamma^2} + 2 \right)\mu s_k } \right\}. $$
By Proposition \ref{Prop: s_k_lower_bound}, we have $s_k \ge \frac{q}{L}$ for all $k\ge 0$, so that
$$ r\ge \min\left\{ \frac{\mu \gamma q}{4L}, \frac{\mu q}{ \frac{2L}{\beta\gamma} + \left( \frac{8}{\beta\gamma^2} + 2 \right)\mu q } \right\}:=\rho. $$
It follows that
$$ E_{k+1} \le (1-\rho)E_k, $$
so that
$$ E_k \le (1-\rho)^k E_0,\quad \forall k\ge 0. $$
As a result,
$$ f(x_k) - f^* \le \frac{E_k}{\gamma t_k^2 s_k} \le \frac{E_0}{\gamma t_k^2 s_k}(1-\rho)^k. $$
Given $x_0=y_0$, \eqref{E: E_0_bound} yields
$$ E_0 \le \frac{1}{2}\| x_0 - x^* \|^2 + \frac{1}{2}\gamma s_0t_0 \left[ (1+\beta)\gamma s_0t_0 - \frac{1}{L} \right] \| \nabla f(x_0) \|^2 + \gamma t_0(t_0-1)s_0 ( f(x_0) - f^* ). $$
Recalling that $s_k \ge \frac{q}{L}$, we obtain
$$ f(x_k) - f^* \le \frac{E_0 L}{\gamma q t_k^2}(1-\rho)^k \le \frac{DL}{t_k^2}(1-\rho)^k, $$
as claimed.
\end{proof}

\begin{remark}
In case $f$ is only convex, we still have
$$ f(x_k) - f^* \le \frac{DL}{t_k^2} \le \mathcal{O}\left( \frac{1}{k^2} \right). $$
Also, in view of Lemma \ref{Lem: E_k_diff_SC}, we have
$$ \frac{\beta\gamma^2 t_k^2 s_k^2}{4} \| \nabla f(x_k) \|^2 \le E_k - E_{k+1},\quad \forall k\ge 0,$$
which implies
$$ \sum_{k=0}^\infty k^2 \| \nabla f(x_k) \|^2 < \infty. $$
\end{remark}

We now examine two particular cases of Theorem \ref{Thm: func_value_SC}, to illustrate concrete convergence rates:

{\bf The case $\gamma=\frac{1}{2}$, $\beta=1$ and $t_0=2$.} Theorem \ref{Thm: func_value_SC} gives:

\begin{corollary}
Let $f:H\to\R$ be $\mu$-strongly convex and $L$-smooth, where $L\ge\mu>0$. Let the sequence $(t_k)_{k\ge 0}$ be given by \eqref{E: t_k} with $m\in(0,1)$ and $t_0= 2$. Consider the algorithm \eqref{Algo: A-AGM}, where $\gamma =\frac{1}{2}$ and Hypothesis \ref{Hypo: s_k} holds with $\omega=\delta=\frac{1}{2}$ and $\beta=1$. Then, given $s_0 \ge \frac{1}{12L}$ and $x_0=y_0$, we have, for every $k\ge 0$,
$$ s_k \ge \frac{1}{12L} \qbox{and} f(x_k) - f^* \le \frac{12DL}{t_k^2}(1-\rho)^k, $$
where
\begin{align*}
\rho &= \min\left\{ \frac{\mu}{96L}, \frac{\mu}{ 48L+34\mu } \right\},\\
D &= \| x_0 - x^* \|^2 + s_0 \left( 2 s_0 - \tfrac{1}{L} \right) \| \nabla f(x_0) \|^2 + 2s_0 ( f(x_0) - f^* ). 
\end{align*}
\end{corollary}

{\bf The case $\gamma=1$, $\beta=\frac{1}{3}$ and $t_0=3$.} Now Theorem \ref{Thm: func_value_SC} becomes:

\begin{corollary}
Let $f:H\to\R$ be $\mu$-strongly convex and $L$-smooth, where $L\ge\mu>0$. Let the sequence $(t_k)_{k\ge 0}$ be given by \eqref{E: t_k} with $m\in(0,1)$ and $t_0= 3$. Consider the algorithm \eqref{Algo: A-AGM}, where $\gamma = 1$ and Hypothesis \ref{Hypo: s_k} holds with $\omega=\delta=\frac{1}{2}$ and $\beta=\frac{1}{3}$. Then, given $s_0 \ge \frac{1}{16L}$ and $x_0=y_0$, we have, for every $k\ge 0$,
$$ s_k \ge \frac{1}{16L},\qbox{and} f(x_k) - f^* \le \frac{8DL}{t_k^2}(1-\rho)^k, $$
where
\begin{align*}
\rho &= \min\left\{ \frac{\mu}{64L}, \frac{\mu}{ 96L + 26\mu } \right\},\\
D &= \| x_0 - x^* \|^2 + 3s_0 \left( 4 s_0 - \tfrac{1}{L} \right) \| \nabla f(x_0) \|^2 + 12 s_0 ( f(x_0) - f^* ). 
\end{align*}
\end{corollary}

%%%%%%%%%%%%%%%%%%%%%%%%%%%%%%%%%%%%%%%%
%%%%%%%%%%%%%%%%%%%%%%%%%%%%%%%%%%%%%%%%
\section{Conclusions}\label{Sec: conclusions}
We have developed an adaptive accelerated gradient method for solving smooth convex optimization problems. The method is free of line search procedures. It provides a convergence guarantee of the iterates and ensures a fast convergence rate $\mathcal{O}\left( \frac{L}{k^2} \right)$ for the function values when the objective function $f$ is convex and $L$-smooth, and a linear convergence rate $\mathcal{O}\left( \frac{L}{k^2}(1-\rho)^k \right)$, with $\rho=\mathcal{O}\left( \frac{\mu}{L} \right)$, in case $f$ is $\mu$-strongly convex.

%%%%%%%%%%%%%%%%%%%%%%%%%%%%%%%%%%%%%%%%
%%%%%%%%%%%%%%%%%%%%%%%%%%%%%%%%%%%%%%%%
% \section*{Acknowledgments}

%%%%%%%%%%%%%%%%%%%%%%%%%%%%%%%%%%%%%%%%
%%%%%%%%%%%%%%%%%%%%%%%%%%%%%%%%%%%%%%%%
\bibliographystyle{siamplain}
\bibliography{myrefs}

\begin{thebibliography}{10}

\bibitem{Armijo_1966}
{\sc L.~Armijo}, {\em Minimization of functions having {L}ipschitz continuous
  first partial derivatives}, Pacific Journal of Mathematics, 16 (1966),
  pp.~1--3.

\bibitem{Attouch_2018}
{\sc H.~Attouch, Z.~Chbani, J.~Peypouquet, and P.~Redont}, {\em Fast
  convergence of inertial dynamics and algorithms with asymptotic vanishing
  viscosity}, Mathematical Programming, 168 (2018), pp.~123--175.

\bibitem{Bao_2023}
{\sc C.~Bao, L.~Chen, and J.~Li}, {\em The global {R}-linear convergence of
  {N}esterov's accelerated gradient method with unknown strongly convex
  parameter}, arXiv:2308.14080v2,  (2023).

\bibitem{Borodich_2025}
{\sc E.~Borodich and D.~Kovalev}, {\em {N}esterov finds {GRAAL}: Optimal and
  adaptive gradient method for convex optimization}, arXiv:2507.09823,  (2025).

\bibitem{Radu_2025}
{\sc R.~I. Boţ, E.~Chenchene, E.~R. Csetnek, and D.~A. Hulett}, {\em
  Accelerating diagonal methods for bilevel optimization: Unified convergence
  via continuous-time dynamics}, arXiv:2505.14389,  (2025).

\bibitem{Radu_2025_iterate}
{\sc R.~I. Boţ, J.~Fadili, and D.-K. Nguyen}, {\em The iterates of
  {N}esterov's accelerated algorithm converge in the critical regimes},
  arXiv:2510.22715,  (2025).

\bibitem{Chambolle_2015}
{\sc A.~Chambolle and C.~Dossal}, {\em On the convergence of the iterates of
  the “fast iterative shrinkage/thresholding algorithm”}, Journal of
  Optimization Theory and Applications, 166 (2015), pp.~968--982.

\bibitem{Ghaderi_2025}
{\sc S.~Ghaderi, M.~Rahimi, Y.~Moreau, and M.~Ahookhosh}, {\em ({A}daptive)
  {S}caled gradient methods beyond locally {H}{\"o}lder smoothness: Lyapunov
  analysis, convergence rate and complexity}, arXiv:2511.10425,  (2025).

\bibitem{Ryu_2025}
{\sc U.~Jang and E.~K. Ryu}, {\em Point convergence of {N}esterov's accelerated
  gradient method: An {AI}-assisted proof}, arXiv:2510.23513,  (2025).

\bibitem{Latafat_2025}
{\sc P.~Latafat, A.~Themelis, L.~Stella, and P.~Patrinos}, {\em Adaptive
  proximal algorithms for convex optimization under local {L}ipschitz
  continuity of the gradient}, Mathematical Programming, 213 (2025),
  pp.~433--471.

\bibitem{Shi_2024}
{\sc B.~Li, B.~Shi, and Y.~Yuan}, {\em Linear convergence of forward-backward
  accelerated algorithms without knowledge of the modulus of strong convexity},
  SIAM Journal on Optimization, 34 (2024), pp.~2150--2168.

\bibitem{Lan_2025}
{\sc T.~Li and G.~Lan}, {\em A simple uniformly optimal method without line
  search for convex optimization}, Mathematical Programming,  (2025).

\bibitem{Malitsky_2020}
{\sc Y.~Malitsky and K.~Mishchenko}, {\em Adaptive gradient descent without
  descent}, in Proceedings of the 37th International Conference on Machine
  Learning, vol.~119 of PMLR, 2020, pp.~6702--6712.

\bibitem{Malitsky_2024}
{\sc Y.~Malitsky and K.~Mishchenko}, {\em Adaptive proximal gradient method for
  convex optimization}, in Advances in Neural Information Processing Systems,
  vol.~37, 2024, pp.~100670--100697.

\bibitem{Nesterov_1983}
{\sc Y.~Nesterov}, {\em A method for solving the convex programming problem
  with convergence rate $\mathcal{O}\bigl(\frac{1}{k^2}\bigr)$}, Soviet
  Mathematics Doklady, 27 (1983), pp.~372--376.

\bibitem{Nesterov_2004}
{\sc Y.~Nesterov}, {\em Introductory Lectures on Convex Optimization: A Basic
  Course}, Springer New York, NY, New York, 2004.

\bibitem{Ma_2025}
{\sc J.~J. Suh and S.~Ma}, {\em An adaptive and parameter-free {N}esterov's
  accelerated gradient method for convex optimization}, arXiv:2505.11670,
  (2025).

\bibitem{Wang_2025_AVD}
{\sc Z.~Wang and J.~Peypouquet}, {\em Fast convex optimization via inertial
  systems with asymptotically vanishing viscosity and {H}essian-driven
  damping}, arXiv:2506.21730,  (2025).

\bibitem{Ma_2025_AdaBB}
{\sc D.~Zhou, S.~Ma, and J.~Yang}, {\em Ada{BB}: Adaptive {B}arzilai-{B}orwein
  method for convex optimization}, Mathematics of Operations Research,  (2025).

\end{thebibliography}

\end{document}